\documentclass[english]{article}
\usepackage{geometry}
\usepackage{booktabs}
\usepackage{float}
\usepackage{mathrsfs}
\usepackage{amsmath}
\usepackage{amssymb}
\usepackage{graphicx}
\usepackage{amsfonts}
\usepackage{bm}
\usepackage{subfigure}
\usepackage{amsthm}
\usepackage{algorithm}
\usepackage{epsfig}
\makeatletter

\@ifundefined{definecolor}{\@ifundefined{definecolor}
 {\@ifundefined{definecolor}
 {\usepackage{color}}{}
}{}
}{}

\usepackage{subfig}\usepackage[all]{xy}

\numberwithin{equation}{section}
\theoremstyle{plain}
\newtheorem{thm}{Theorem}
\newtheorem{rmk}{Remark}
\newtheorem{lem}{Lemma}

\newtheorem{cor}{Corollary} 
\newcommand{\bx}{{\bf x}}

\newcommand{\cF}{{\cal F}}
\newcommand{\cG}{{\cal G}}
\newcommand{\cH}{{\cal H}}

\newcommand{\bX}{{\bf X}}
\newcommand{\bY}{{\bf Y}}
\newcommand{\wht}{\widehat}
\newcommand{\wtd}{\widetilde}

\newcommand{\wX}{\widetilde X}

\newcommand{\wbX}{\widetilde {\bf X}}

\newcommand{\bk}{{\bf k}}

\newcommand{\bi}{{\bf i}}
\newcommand{\srp}{\stackrel{p}{\rightarrow}}
\newcommand{\mZ}{\mathsf{Z}}

\usepackage{babel}\date{}

\usepackage{babel}

\makeatother

\usepackage{babel}

\begin{document}

\begin{center}

{\Large \textbf{Theory of Parallel Particle Filters for Hidden Markov Models}}

\vspace{0.5cm}

BY HOCK PENG CHAN, CHIANG WEE HENG \& AJAY JASRA

{\footnotesize Department of Statistics \& Applied Probability,
National University of Singapore, Singapore, 117546, SG.}\\
{\footnotesize E-Mail:\,}\texttt{\emph{\footnotesize stachp@nus.edu.sg}},  \texttt{\emph{\footnotesize a0068363@nus.edu.sg}}, \texttt{\emph{\footnotesize staja@nus.edu.sg}}
\end{center}

\begin{abstract}
The objective of this article is to study the asymptotic behavior of 
a new particle filtering approach in the context of
hidden Markov models (HMMs). In particular, we develop an algorithm 
where the latent-state sequence is segmented into multiple shorter portions,
with an estimation technique based upon a separate particle filter in each portion.
The partitioning
facilitates the use of parallel processing. Based upon this approach, we introduce new estimators of the latent states and likelihood
which have similar or better variance properties compared to estimators derived from standard particle filters. 
Moreover due to parallelization 
there is less wall-clock computational time.
We show that
the likelihood function estimator is unbiased, prove central limit theorem convergences of estimators, and 
provide consistent in-sample 
estimation of the asymptotic variances. 
The theoretical analyses, supported by a numerical study, show
that the segmentation reduces the variances in smoothed latent-state estimators, in addition to the savings in wall-clock time.
\end{abstract}

\section{Introduction}

Hidden Markov models are a flexible class of statistical models that are applied
in a wide variety of applications such as bioinformatics, economics, engineering and finance;
see \cite{cappe} for an introduction.
Mathematically a HMM corresponds to a pair of discrete-time processes $X_t\in\mathsf{X}$, $Y_t\in\mathsf{Y}$, with the observed $Y_t$
conditionally independent given $X_t$ and the unobserved $X_t$ obeying a first-order Markov chain
\begin{equation} \label{hmm}
(X_t|X_{t-1}=x) \sim P_\theta(\cdot|x), \quad (Y_t|X_t=x) \sim G_\theta(\cdot|x), \quad t \geq 1
\end{equation}
with densities, $p_\theta$ and $g_\theta$ with respect to 
dominating measures on their state-spaces and $\theta$ is 
a static parameter. 

From an inferential perspective, we are interested in the conditional distribution
of $X_t$ given all the observations up to and perhaps after time $t$. 
This has a wide-range of interpretations, particularly
in real-time applications. In addition there is much practical interest in the calculation
of the likelihood of the observations, for model comparison and parameter
estimation. The difficulty with the afore-mentioned objectives is that the exact computation
of the conditional distribution or likelihood is typically not
possible, as often the high-dimensional integrals that it depends on are often
difficult to evaluate. In practice Monte
Carlo-based numerical methods are adopted, in particular the method
of particle filters or equivalently sequential Monte Carlo (SMC);
see \cite{doucet} for an overview.

SMC methods can be described as a collection of techniques that approximate a sequence of distributions, known up to normalizing constants, of increasing dimensions, and
are often applied to HMMs.
SMC methods combine  importance sampling and resampling to approximate distributions. The idea is to introduce a sequence of proposal densities and sequentially simulate a collection of $K>1$ samples, termed particles, in parallel from these proposals. In most scenarios
it is not possible to use the distribution of interest as a proposal. 
Therefore one must correct for the discrepancy between proposal and target via importance weights. 
As the variance of these importance weights can potentially increase exponentially with algorithmic time, resampling is applied to control it. 
Resampling consists of sampling with replacement from the current samples using the weights and then resetting them to $K^{-1}$. The theoretical properties of SMC 
with regards to their convergence as $K$ grows
are well-studied; see \cite{CL13,Cho04,Del04,DM08,Kun05}. 

In recent years,
the applicability of SMC techniques has been enhanced by parallel computation; see 
\cite{lee}. One of the main bottlenecks in the application of parallel computation to
SMC methods is the resampling step, a major requirement for the method to be efficient.
This has led to a number of researchers investigating methodologies 
that reduce the degree of interaction in
SMC algorithms; see \cite{lee1,lindsten,wh}. 
This work is complementary to the afore-mentioned references,
and is a methodology designed to assist in the parallelization of SMC algorithms, while attempting to retain 
their attractive properties.
Our objective is to study the asymptotic behavior of HMM estimators
when the latent-state sequence is segmented into multiple shorter portions,
by applying an estimation technique based upon a separate particle filter in each portion.
The partitioning
facilitates the use of parallel processing. Based upon this approach, we introduce new SMC-based estimators of the 
latent states (that is, expectations w.r.t.~the filter and smoother) and likelihood
with similar or better variance properties compared to standard SMC estimators, but due to parallelization can
be calculated in less wall-clock computational time. In particular we show:

\begin{itemize}
\item unbiasedness of our likelihood estimate,
\item central limit convergences of the likelihood and latent-state estimates,
\item consistent estimation of asymptotic variances.
\end{itemize}

Our likelihood estimates can be used in conjunction with
recent advances in SMC methodology
in which particle filter processing is just one component of a two-layered
process when learning $\theta$ in a Bayesian manner: particle MCMC (PMCMC) \cite{ADH10}, SMC$^2$ \cite{CJP12},
MCMC substitution \cite{CL14}. That is, our procedure can be routinely used in the context of these works.
In principle,
there is no need to break up the observation sequence into
strictly disjoint segments, it can be advantageous to include additional observations
at the edges of the segments to smooth out the joining of the sample paths. 
We shall illustrate this in the numerical study in Section 3.3.2. 

We describe the algorithm and estimators in Secton 2 and the asymptotic theory in Section 3,
with an illustration of variance reduction in smoothed latent-state estimators. 
We discuss refinements in Section 4. The technical proofs are consolidated in the Appendix.

\section{Independent particle filters for segmented data}\label{sec:alg_note}

Let $\bY_U =(Y_1, \ldots, Y_U)$ for some $U > 1$. 
As the observation sequence is conditionally
independent given the latent-state sequence, the density 
of $\bX_U := (X_1, \ldots, X_U)$ conditioned on $\bY_U$ is given by
\begin{equation} \label{tp}
p_{\theta}(\bx_U|\bY_U) = \prod_{t=1}^U [p_\theta(x_t|x_{t-1}) g_\theta(Y_t|x_t)] \Big/ \lambda(\theta),
\end{equation}
where $\lambda(\theta) [=\lambda(\bY_U|\theta)]$ is the likelihood function that
normalizes $p_\theta(\cdot|\bY_U)$ so that it integrates to 1,
and $p_\theta(x_t|x_{t-1})$ for $t=1$ denotes $p_\theta(x_1)$.

Let $\bx_t = (x_1, \ldots, x_t)$ and 
let $q_t(\cdot|\bx_{t-1})$ be an importance density of $X_t$, 
with $q_t(\cdot|\bx_{t-1})$ for $t=1$ denoting
$q_1(\cdot)$. We shall require that $q_t(x_t|\bx_{t-1}) > 0$ whenever $p_\theta(x_t|x_{t-1})>0$. 
For notational simplicity, we assume that $U=MT$
for positive integers $M$ and $T$, so that the latent-state sequence
can be partitioned neatly into $M$ subsequences of equal length $T$.
We shall operate $M$ particle filters independently,
with the $m$th particle filtering generating sample paths of 
$\bX_{m,mT}$, where $\bX_{m,t} = (X_{(m-1)T+1}, \ldots, X_t)$. 
Due to the independent nature of the particle filters, we require that for $(m-1)T < t \leq mT$, $q_t(\cdot|\bx_{t-1})$
does not depend on $\bx_{(m-1)T}$. 
We can thus express $q_t(\cdot|\bx_{t-1})$ as $q_t(\cdot|\bx_{m,t-1})$. 

Let $w_t(\bx_t)$ be the positive resampling weights of a sample path $\bx_t$, and
again due to the independent nature of the particle filters, we shall require that for $(m-1)T < t \leq mT$, 
$w_t(\bx_t)$ does not depend on $\bx_{(m-1)T}$, 
and express $w_t(\bx_t)$ also as $w_t(\bx_{m,t})$. 

\subsection{Approach}

We shall apply 
standard multinomial resampling at every stage, as proposed in the seminal paper 
\cite{GSS93}; this is not necessary from a methodological point of view, but we will analyze this case. 
In the case of a single particle filter, it is common to adopt 
$$
q_t(x_t|\bx_{t-1}) = p_\theta(x_t|x_{t-1}), \quad\quad w_t(x_t) = g_{\theta}(Y_t|x_t), 
$$
but it need not be the case, and we can in general let
\begin{equation} \label{wtxt}
w_t(\bx_t) = \frac{g_{\theta}(Y_t|x_t)p_\theta(x_t|x_{t-1})}{q_t(x_t|\bx_{t-1})}.
\end{equation}
Therefore the single particle filter targets (up-to proportionality), after resampling at stage $t$:
$$\prod_{u=1}^t [g_\theta(Y_u|x_u) p_\theta(x_u|x_{u-1})].
$$
For the parallel particle filters, the $m$th particle filter, after resampling at stage $(m-1)T < t \leq mT$, targets
\begin{eqnarray} \label{eq:second_target}
\pi_{m,t}(\bx_{m,t}) & \propto &  
r_m(x_{(m-1)T+1})g_{\theta}(Y_{(m-1)T+1}|x_{(m-1)T+1}) \\ \nonumber
& & \quad \times \prod_{u=(m-1)T+2}^t [g_{\theta}(Y_u|x_u)p_\theta(x_u|x_{u-1})],
\end{eqnarray}
where $r_m(\cdot)$ is a positive probability on $\mathsf{X}$ which can be evaluated up to a constant,
and is independent from output of the others filters; sensible choices of $r_m(\cdot)$  
are suggested in Section 3.3.2. For $m=1$, we can simply let $r_m(x_1) = 
p_\theta(x_1)$. 
The forms of the target and proposal mean that the particle 
filters can be run in parallel with each other.
Therefore $w_t$ has the form given in (\ref{wtxt}), 
with the exceptions that when $t=(m-1)T+1$,
$p_\theta(x_t|x_{t-1})$ is replaced by $r_m(x_t)$.

The particle filter approach is given below. It is remarked that some of the notations, for example $H_t^k$,
are not needed to run the particle filter but
will help to facilitate the subsequent theoretical analysis.
For $1 \leq m  \leq M$:

\smallskip
{\bf Particle filter $\mathbf{m}$ (PF$\mathbf{m}$)}. Recursively at stages 
$t=(m-1)T+1,\ldots,mT$: 

\begin{enumerate}
\item {\it Importance sampling}. For $1 \leq k \leq K$, sample
$\wX_t^k \sim q_t(\cdot|\bX_{m,t-1}^k)$ and let $\wbX_{m,t}^k = \wX_t^k$ if $t=(m-1)T+1$, 
$\wbX_{m,t}^k = (\bX_{m,t-1}^k,
\wX_t^k)$ otherwise.

\item {\it Resampling}. Generate i.i.d. $B_t^1, \ldots,$ $B_t^K$ [$B(1), \ldots, B(K)$ for short] such that
\begin{equation} \label{R1}
P \{ B(1)=j \} = w_t(\wbX_{m,t}^j)/(K \bar w_t) (= W_t^j), 
\end{equation} 
where $\bar w_t = K^{-1} \sum_{k=1}^K
w_t(\wbX_{m,t}^k)$.

\item {\it Updating}. Let $(\bX_{m,t}^k, A_{m,t}^k) = 
(\wbX_{m,t}^{B(k)}, A_{m,t-1}^{B(k)})$, 
\begin{equation} \label{PF1}
\wtd H_{m,t}^j = H_{m,t-1}^j/(K W_t^j) \mbox{ and } H_{m,t}^k = \wtd H_{m,t}^{B(k)}, \quad
1 \leq k \leq K,
\end{equation}
with the conventions $A_{m,(m-1)T}^k=k$, $H_{m,(m-1)T}^k=1$.
\end{enumerate}

\begin{rmk}
There are other procedures that use two or more particle filters to perform estimation. For instance, \cite{persing} introduces
a method based upon generalized two-filter smoothing. However, that approach is restricted to two particle filters that 
run forwards and backwards,
and requires the choice of pseudo-densities, which may be more difficult than the choice of $r_m(\cdot)$.
The approach of \cite{verge} uses multiple particle filters to perform estimation, but is different from the ideas in this article. Typically that approach will
run filters in parallel on the same target and allow the filters themselves to interact. In our approach, we 
are able to reduce variability (relative to one particle filter) of estimates
by segmentation, which is possibly not achieved in \cite{verge}. 
\end{rmk}

\subsection{Notations}

Set $\eta_0=1$ and assume that
\begin{equation} \label{eta}
\eta_t := E_q \Big[ \prod_{u=1}^t w_u(\bX_u) \Big] < \infty \mbox{ for } 1 \leq t \leq U,
\end{equation}
where for (integrable) $\varphi:\mathsf{X}^{t}\rightarrow\mathbb{R}$,
\begin{equation} \label{Eqphi}
E_q \varphi(\bX_t) = \int_{\mathsf{X}^{t}}
\varphi(\bx_t)
\Big[ \prod_{u=1}^t q_u(x_u|\bx_{u-1}) \Big] d \bx_t.
\end{equation}

Consider $(m-1)T < t \leq mT$. Define $\eta_{m,t} = \eta_t/\eta_{(m-1)T}$,
and let
\begin{equation}  
\label{h1} h_t(\bx_t) = \eta_t \Big/ \prod_{u=1}^t w_u(\bx_u), \quad
h_{m,t}(\bx_{m,t}) = \eta_{m,t} \Big/ 
\prod_{u=(m-1)T+1}^t w_u(\bx_{m,u}).
\end{equation}
By (\ref{R1})--(\ref{h1}), 
\begin{equation} \label{Ht}
H_{m,t}^k = \Big( \frac{\bar w_{(m-1)T+1} \cdots \bar w_t}{\eta_{m,t}} \Big)
h_{m,t}(\bX_{m,t}^k).
\end{equation}

Let $\mZ^m (=\mZ_K^m)= \{ (k(1), \ldots, k(m)): 1 \leq k(n) \leq K$ for $1 \leq n \leq m \}$.
For $\bk \in \mZ^m$, let
\begin{eqnarray} \label{XH}
\wbX_t^{\bk} & = & (\bX_{1,T}^{k(1)}, \ldots, \bX_{m-1,(m-1)T}^{k(m-1)},
\wbX_{m,t}^{k(m)}), \quad 
\wtd H_t^{\bk} = \Big( \prod_{n=1}^{m-1} H_{n,nT}^{k(n)} \Big)
\wtd H_{m,t}^{k(m)}, \\ \nonumber
\bX_t^{\bk} & = & (\bX_{1,T}^{k(1)}, \ldots, \bX_{m-1,(m-1)T}^{k(m-1)},
\bX_{m,t}^{k(m)}), \quad 
H_t^{\bk} = \Big( \prod_{n=1}^{m-1} H_{n,nT}^{k(n)} \Big)
H_{m,t}^{k(m)}. 
\end{eqnarray}
Thus analogous to (\ref{Ht}),
\begin{equation} \label{Htk}
H_t^{\bk} = \Big( \frac{\bar w_1 \cdots \bar w_t}{\eta_t} \Big)
h_t(\bX_t^{\bk}).
\end{equation}

The notation
$A_{m,t}^k$ refers to the first-generation ancestor of $\bX_{m,t}^k$ (or $\wbX_{m,t+1}^k$). 
That is $A_t^k=j$ if the first component of $\bX_{m,t}^k$ 
is $\wtd X_{m-1)T+1}^j$. This
ancestor tracing is exploited in Sections 3.3.3 for standard error approximations of the estimates.
Finally $N(\mu,\sigma^2)$ denotes the normal distribution with mean $\mu$ and variance $\sigma^2$.

\section{Estimation theory}\label{sec:estimation}

We are interested in the estimation of the likelihood $\lambda(\theta)$, and also of 
$\psi_U := E_p[\psi(\bX_U)|\bY_U]$ for some real-valued measurable function $\psi$.
Here $E_p$ denotes expectation under the HMM (\ref{hmm}). The estimation of $\lambda(\theta)$ falls under
the canonical case; the theory is given in Section \ref{sec:can_nc}. The estimation of $\psi_U$ falls
under the non-canonical case; the theory is given in 
Section~\ref{sec:nccan_at}. 

\subsection{Estimates and Remarks}

\subsubsection{Canonical Case}\label{sec:can_case}

Define the function
\begin{equation} \label{Lx}
L(\bx_U) = p_\theta(\bx_U|\bY_U) \Big/ \prod_{t=1}^U q_t(x_t|\bx_{t-1}),
\end{equation}
where $p_\theta(\bx_U|\bY_U)$ is as \eqref{tp}.
The estimator of $\psi_U$ in the canonical case, 
which we will prove is unbiased, is given by 
\begin{equation} \label{hatpsi}
\wht \psi_U = K^{-M} \sum_{\bk \in \mZ^M} L(\bX_U^{\bk}) \psi(\bX_U^{\bk}) H_U^{\bk}.
\end{equation}
By (\ref{tp}) and (\ref{Lx}), $\lambda(\theta)$ appears in the denominator on 
the R.H.S. of (\ref{hatpsi}). This does not pose a problem in
the estimation of $\lambda(\theta)$, for by setting $\psi \equiv \lambda(\theta)$, we cancel out $\lambda(\theta)$.
We define $\wht \lambda(\theta)$ to be the estimator obtained this way, that is
\begin{equation} \label{lhat}
\wht \lambda(\theta)/\lambda(\theta) = K^{-M} \sum_{\bk \in \mZ^M} 
L(\bX_U^{\bk}) H_U^{\bk}.
\end{equation}

To further understand this estimate, we rewrite (\ref{lhat}) as
\begin{equation}
\wht \lambda(\theta) = 
\Big( \prod_{t=1}^U \bar w_t \Big) \prod_{m=2}^M 
\Big( K^{-2} \sum_{k=1}^K \sum_{\ell=1}^K \frac{p_{\theta}(X_{(m-1)T+1}^{\ell}|
X_{(m-1)T}^{k})}{r_m(X_{(m-1)T+1}^{\ell})} \Big),
\label{eq:lam_clar}
\end{equation}
where $X_{(m-1)T+1}^{\ell}$ here refers to the first component of $X_{m,mT}^{\ell}$.
A heuristic justification is as follows. 
For expositional purposes, let us consider the simplest case of $M=2$. 
The final term on the R.H.S.~of \eqref{eq:lam_clar}, the double summation,
is an SMC estimate of the ratio, up to a constant,
of the actual target of interest \eqref{tp}, 
and the normalized target (\ref{eq:second_target}) that is sampled by the two particle filters. That is as $K\rightarrow
\infty$, we would like to obtain convergence to
$$
\int_{\mathsf{X}^U}  
\frac{p_{\theta}(x_{T+1}|
x_{T})}{r_2(x_{T+1})} \pi_{1,T}(\bx_T)\pi_{2,U}(\bx_{2,U}) d\bx_U.
$$
The term $\big( \prod_{t=1}^U \bar w_t \big)$ will converge in probability to the normalizing constants of $\pi_{2,U}(\bx_{2,U})$.

The expression \eqref{eq:lam_clar} also suggests a good choice of $r_2(\cdot)$. If we take
$$
r_2(x) = K^{-1} \sum_{k=1}^K p_{\theta}(x|
x_{T}^{k}),
$$
then the double sum on the R.H.S.~of \eqref{eq:lam_clar} is exactly 1; that is, it does not
contribute to the variance of the estimate. The choice above is exactly the SMC approximation
of the predictor. However, the choice suggested above is not reasonable in that it circumvents
the parallel implementation of the two filters. However we should thus choose  $r_2(\cdot)$ to approximate the predictor.
This will be illustrated in Section 3.3.2. 
We will also discuss in Section 4 how
subsampling can be used to reduce the computational cost of calculating $\wht \lambda(\theta)$ to 
$\mathcal{O}(K)$. 

\subsubsection{Non-Canonical Case}

In the case of latent-state estimation under the non-canonical case, 
the unknown $\lambda(\theta)$ inherent in (\ref{hatpsi}) is replaced by 
$\wht \lambda(\theta)$, that is we divide the R.H.S. of (\ref{hatpsi}) and (\ref{lhat}) to obtain the estimator 
\begin{equation} \label{tpsi}
\wtd \psi_U = \sum_{\bk \in \mZ^M} L(\bX_U^{\bk}) \psi(\bX_U^{\bk}) H_U^{\bk} \Big/ \sum_{\bk \in \mZ^M} L(\bX_U^{\bk}) H_U^{\bk}.
\end{equation}
We can rewrite the above estimate in a standard form seen in the literature, 
and reduce the cost of computation to $\mathcal{O}(K)$. For example, if there is only one particle filter and we select $w_t$ to satisfy
(\ref{wtxt}), then 
the estimate reduces to $K^{-1} \sum_{k=1}^K \psi(\bX_U^k)$,
which is the standard estimate in the literature.

\subsection{Unbiasedness and CLT under the canonical case}\label{sec:can_nc}

Let $f_0 = \psi_U$ and define, 
\begin{equation} \label{fu}
f_t(\bx_t) = E_q[\psi(\bX_U) L(\bX_U)|\bX_t=\bx_t], \quad 1 \leq t \leq U,
\end{equation}
where $E_q$ denotes expectation with respect to the importance densities $q_t$ as defined in (\ref{Eqphi}).
There is no resampling involved under $E_q$. 
Let $\#_t^k$ denotes the number of copies of $\wbX_t^k$ generated from
$(\wbX_t^1, \ldots, \wbX_t^K)$ to form $(\bX_t^1, \ldots, \bX_t^K)$. Thus conditionally,
$(\#_t^1, \ldots, \#_t^K) \sim$ Multinomial $(K, W_t^1, \ldots, W_t^K)$. Let $\cF_{2t-1}$ and
$\cF_{2t}$ denote the $\sigma$-algebras generated by all random variables just before 
and just after resampling respectively, 
at the $t^{\textrm{th}}$ stage. In the case of $(m-1)T < t \leq mT$ for $m>1$, these $\sigma$-algebras include 
all random variables generated by PF1 to PF$(m-1)$.
Let $E_K$ denote expectation with respect to $K$ sample paths generated in each particle filter.

\begin{thm} \label{thm1} Define, for $(m-1)T < t \leq mT$ and $1 \leq j \leq K$,
\begin{eqnarray} \label{e2}
\epsilon_{2t-1}^j & = & K^{-m+1} \sum_{\bk \in \mZ^m: A_{m,t-1}^{k(m)}=j} [f_t(\wbX_t^{\bk})-f_{t-1}(\bX_{t-1}^{\bk})] H_{t-1}^{\bk}, \\
\nonumber \epsilon_{2t}^j & = & K^{-m+1} \sum_{\bk \in \mZ^m: A_{m,t-1}^{k(m)}=j} (\#_t^{k(m)}-KW_t^{k(m)}) f_t(\bX_t^{\bk}) 
\wtd H_t^{\bk}.
\end{eqnarray}
Then for each $j$ and $m$, $\{ \epsilon_u^j, \cF_u, 2(m-1)T < u 
\leq 2mT \}$ is a martingale difference sequence, and
\begin{equation} \label{MD}
K(\wht \psi_U-\psi_U) = \sum_{m=1}^M \sum_{j=1}^K (
\epsilon_{2(m-1)T+1}^j + \cdots + \epsilon_{2mT}^j).
\end{equation}
Therefore $E_K \wht \psi_U = \psi_U$.
\end{thm}

\begin{proof}
Since $\#_t^k \sim$ Binomial$(K,W_t^k)$ when conditioned on $\cF_{2t-1}$,
by the tower law of conditional expectations,
\begin{eqnarray*}
E_K(\epsilon_{2t-1}^j|\cF_{2t-2}) & = & K^{-m+1} \sum_{\bk \in \mZ^m: A_{t-1}^{k(m)}=j} \{ E_K[f_t(\wbX_t^{\bk}|\cF_{2t-2}]
-f_{t-1}(\bX_{t-1}^{\bk}) \} H_{t-1}^{\bk} = 0, \cr
E_K(\epsilon_{2t}^j|\cF_{2t-1}) & = & K^{-m+1} \sum_{\bk \in \mZ^m: A_{t-1}^{k(m)}=j} [E_K(\#_t^{k(m)}|\cF_{2t-1})-KW_t^{k(m)}]
f_t(\wbX_t^{\bk}) \wtd H_t^{\bk} = 0,
\end{eqnarray*}
therefore $\{ \epsilon_u^j, \cF_u, 2(m-1)T < u 
\leq 2mT \}$ are indeed martingale difference sequences. 

It follows from (\ref{PF1}) and (\ref{XH}) that
\begin{eqnarray*}
& & \sum_{\bk \in \mZ^m:A_{t-1}^{k(m)}=j} (\#_t^{k(m)}-KW_t^{k(m)}) f_t(\wbX_t^{\bk}) \wtd H_t^{\bk} \cr
& = & \sum_{\bk \in \mZ^m: A_t^{k(m)}=j}
f_t(\bX_t^{\bk}) H_t^{\bk}-\sum_{\bk \in \mZ^m:A_{t-1}^{k(m)}=j} f_t(\wbX_t^{\bk}) H_{t-1}^{\bk},
\end{eqnarray*} 
therefore by (\ref{e2}) and the cancellation of terms in a telescoping series,
\begin{eqnarray} \label{esum}
\sum_{u=2(m-1)T+1}^{2mT} \epsilon_u^j & = & K^{-m+1} \sum_{\bk \in
\mZ^m:A_{m,mT}^{k(m)}=j} f_{mT}(\bX_{mT}^{\bk}) H_{mT}^{\bk} \\ \nonumber
& & \quad - K^{-m+1} \sum_{\bk \in
\mZ^{m-1}} f_{(m-1)T}(\bX_{(m-1)T}^{\bk}) H_{(m-1)T}^{\bk}.  
\end{eqnarray}
Therefore 
\begin{eqnarray} \label{exp1}
\sum_{j=1}^K \Big( \sum_{u=2(m-1)T+1}^{2mT} \epsilon_u^j \Big) & = &  
K^{-m+1} \sum_{\bk \in
\mZ^m} f_{mT}(\bX_{mT}^{\bk}) H_{mT}^{\bk} \\ \nonumber
& & \quad - K^{-m+2} \sum_{\bk \in
\mZ^{m-1}} f_{(m-1)T}(\bX_{(m-1)T}^{\bk}) H_{(m-1)T}^{\bk}. 
\end{eqnarray}
By (\ref{hatpsi}), the identity (\ref{MD}) follows from adding (\ref{exp1}) over
$1 \leq m \leq M$, keeping in mind that
$f_0 = \psi_U$ and $f_U(\bx_U) = L(\bx_U) \psi(\bX_U)$. 
\end{proof}

The martingale difference expansion (\ref{MD}) is for the purpose of standard error estimation,
see Section 3.3.3. An alternative expansion, for the purpose of CLT theory 
in the spirit of \cite[Chapter 8]{Del04}, is formed from 
the martingale difference sequence 
$\{ (Z_u^1, \ldots, Z_u^K): 2(m-1)T < u \leq 2mT \}$, where 
\begin{eqnarray} \label{Z2}
Z_{2t-1}^k & = & K^{-m+1} \sum_{\bk \in \mZ^m: k(m)=k} 
[f_t(\wbX_t^{\bk})
-f_{t-1}(\bX_{t-1}^{\bk})] H_{t-1}^{\bk}, \\
\nonumber Z_{2t}^k & = & K^{-m+1} \sum_{\bk \in \mZ^m: k(m)=k} f_t(\bX_t^{\bk})
H_t^{\bk} - K^{-m+1} \sum_{\bk \in \mZ^m} W_t^k f_t(\wbX_t^{\bk}) 
\wtd H_t^{\bk}.
\end{eqnarray}
Analogous to (\ref{MD}), we have the expansion
\begin{equation} \label{MD2}
K(\wht \psi_U-\psi_U) = \sum_{u=1}^{2U} (Z_u^1 + \cdots + Z_u^K),
\end{equation}
 from which we can also conclude that $\wht \psi_U$ is unbiased. 

The technical difficulties in working with (\ref{Z2}) to prove the CLT
is considerably more involved compared to the standard single particle filter, as
there is now a sum over a multi-dimensional space. 
Therefore let us provide some intuitions first, 
focusing on the key arguments in the extension of the CLT to $M=2$ segments.

For $t>T$, let 
$$f_{2,t}(\bx_{2,t}) = E_q[f_t(\bX_t)|\bX_{2,t}=\bx_{2,t}].
$$
By a `law of large numbers' argument, see Lemma \ref{lem2} in Appendix A.1,
\begin{eqnarray} \label{ave} 
K^{-1} \sum_{k=1}^K f_t(\wbX_t^{k \ell}) H_T^k & \doteq & f_{2,t}(\wbX_{2,t}^{\ell}), \\ \nonumber
K^{-1} \sum_{k=1}^K f_t(\bX_t^{k \ell}) H_T^k & \doteq & f_{2,t}(\bX_{2,t}^{\ell}).
\end{eqnarray}
Therefore by (\ref{XH}) and (\ref{Z2}), $Z_u^{\ell} \doteq Z_{2,u}^\ell$ for $2T < u \leq 2U$, where
\begin{eqnarray} \label{Z22}
Z_{2,2t-1}^{\ell} & = & [f_{2,t}(\wbX_{2,t}^\ell)-f_{2,t-1}
(\bX_{2,t-1}^{\ell})] H_{2,t-1}^{\ell}, \\ \nonumber
Z_{2,2t}^{\ell} & = & f_{2,t}(\bX_{2,t}^{\ell}) H_{2,t}^{\ell}
-\sum_{j=1}^K W_t^j f_{2,t}(\wbX_{2,t}^j) \wtd H_{2,t}^j.
\end{eqnarray}
We now have a martingale difference sequence $\{ (Z_u^1, \ldots, Z_u^K),
\cF_v, 1 \leq u \leq 2T \}$ that depends on the outcomes of PF1 only, and another sequence
$\{(Z_{2,u}^1, \ldots, Z_{2,u}^K),$ $\cG_u, 2T < u \leq 2U \}$ that depends on
the outcomes of PF2 only, where $\cG_{2t-1}$ and $\cG_{2t}$ denote the $\sigma$-algebras generated by random
variables in PF2 only, just before and just after resampling respectively, at
stage~$t$. Moreover,
\begin{equation} \label{Zapprox}
K(\wht \psi_U - \psi_U) \doteq \sum_{u=1}^{2T} \Big( \sum_{k=1}^K Z_u^k \Big)
+ \sum_{u=2T+1}^{2U} \Big( \sum_{\ell=1}^K Z_{2,u}^\ell \Big).
\end{equation}
Therefore subject to neligible error in \eqref{Zapprox}, 
$\sqrt{K}(\wht \psi_U-\psi_U)$ is asymptotically normal, with variance the
sum of the variance components due to each particle filter.

More generally in the case of $M$ independent particle filters, define
\begin{equation} \label{f2u}
f_{m,t}(\bx_{m,t}) = E_q[f_t(\bX_t)|\bX_{m,t}=\bx_{m,t}],
\end{equation}
and recall the definition of $h_{m,t}$ in (\ref{h1}).

\begin{thm} \label{thm2} Let $\sigma^2 = \sum_{u=1}^{2U} \sigma_u^2$, where for $(m-1)T <  t \leq mT$,
\begin{eqnarray} \label{v2}
\sigma_{2t-1}^2 & = & E_q \{ [f_{m,t}^2(\bX_{m,t})-f_{m,t-1}^2(\bX_{m,t-1})] 
h_{m,t-1}(\bX_{m,t-1}) \}, \\ \nonumber 
\sigma_{2t}^2 & = & E_q \{ [f_{m,t}(\bX_{m,t}) h_{m,t}(\bX_{m,t})-f_0]^2/h_{m,t}(\bX_{m,t}) \}. 
\end{eqnarray}
Assume that $E_q \{ f_t^2(\bX_t) [h_t(\bX_t) + h_{t-1}(\bX_t)] \} < \infty$ for 
$1 \leq t \leq U$. Then $\sigma^2 < \infty$ and 
\begin{equation} \label{CLT}
\sqrt{K} (\wht \psi_U - \psi_U) \Rightarrow N(0,\sigma^2) \mbox{ as } K \rightarrow \infty.
\end{equation}
\end{thm}

\subsection{Asymptotic theory in the non-canonical case}\label{sec:nccan_at}

In the non-canonical case, the estimator $\wtd \psi_U$, see (\ref{tpsi}), 
can be approximated by $\wht \psi_U^c$, an unbiased estimator under the canonical case of the centered function
$$\psi^c(\bx_U) := \psi(\bx_U) - \psi_U.
$$
Therefore, analogous to (\ref{fu}) and (\ref{f2u}), we define $f_{m,(m-1)T}^c(\bx_t) = 0$ and
\begin{equation} \label{f0mt}
f_{m,t}^c(\bx_t) = E_q[\psi^c(\bX_U) L(\bX_U)|\bX_{m,t} = \bx_{m,t}], \quad (m-1)T < t \leq mT. 
\end{equation}
The corollary below then follows from Theorem \ref{thm2}. 

\begin{cor} \label{cor1} Let $\sigma^2_c = \sum_{u=1}^{2U} \sigma_{c,u}^2$, where for $(m-1)T < t \leq mT$, 
\begin{eqnarray} \label{cor1.2}
\sigma_{c,2t-1}^2 & = & E_q ( \{ [f_{m,t}^c(\bX_{m,t})]^2-[f_{m,t-1}^c(\bX_{m,t-1})]^2 \}  
h_{m,t-1}(\bX_{m,t-1})), \\ \nonumber  
\sigma_{c,2t}^2 & = & E_q \{ [f_{m,t}^c(\bX_{m,t})]^2 h_{m,t}(\bX_{m,t}) \}.
\end{eqnarray}
Under the assumptions of Theorem {\rm \ref{thm2}},  
\begin{equation} \label{CLT2}
\sqrt{K} (\wtd \psi_U - \psi_U) \Rightarrow N(0,\sigma^2_c) \mbox{ as } K \rightarrow \infty.
\end{equation}
 \end{cor}

In Section 3.3.3, we show how $\sigma_c^2$ can be estimated in-sample,
and discuss the implications in particle size allocation. Before that, we shall illustrate, in Sections 3.3.1 and 3.3.2,
the advantage of segmentation in providing stability to smoothed latent-state estimations, that is
estimation of $E_p(X_u|\bY_U)$ for $u < U$, with $u$ fixed as $U \rightarrow \infty$.

\subsubsection{Example} 

Consider the linear time-series
\begin{equation} \label{lts}
X_t = aX_{t-1}+\epsilon_t, \quad Y_t = X_t + \eta_t, 
\end{equation}
with $0 < a < 1$, $\epsilon_t \sim N(0, (1-a^2)\sigma_X^2)$ and $\eta_t \sim N(0,\sigma_Y^2)$. 
Let $\theta = (a,\sigma_X^2, \sigma_Y^2)$. 

We shall illustrate on this simple example the advantage of parallel particle filters in smoothed estimation of $X_t$. 
Consider firstly the segmented particle filter with
$T=1$ and $q_t(x_t|\bx_{t-1}) = p_\theta(x_t|\bY_U)$. Let $w_t \equiv 1$, 
therefore $h_{t,t} \equiv 1$ (recall also that $h_{t,t-1} \equiv 1$) for all $t$.
Consider $\psi(\bx_U)=x_u$ for some $1 \leq u < U$. By (\ref{Lx}) and (\ref{f0mt}),  
$$f_{t,t}^c(\bx_t) = E_p(X_u|X_t=x_t, \bY_U)-E_p(X_u|\bY_U).
$$
Since $f_{t,t-1}^c \equiv 0$, therefore by (\ref{cor1.2}), 
\begin{equation} \label{ex1}
\sigma_{c,2t}^2 = \sigma_{c,2t-1}^2 = {\rm Var}_p(E_p(X_u|X_t,\bY_U)|\bY_U) 
\leq {\rm Var}_p(E_p(X_u|X_t)) = a^{|u-t|} \sigma_X^2,
\end{equation}
with the inequality in (\ref{ex1}) following from say eq.(2.5.3) of \cite{And84}.
Therefore 
$$\sigma_c^2 = \sum_{t=1}^U (\sigma_{c,2t-1}^2+\sigma^2_{c,2t}) \leq \Big( \frac{4}{1-a}-2 \Big) \sigma_X^2.
$$
Corollary \ref{cor1} says that, in this example,  
$\sigma_c^2$ is uniformly bounded and hence
the estimation of $X_u$ is stable as $U \rightarrow \infty$.

For the standard particle filter with no latent-state sequence segmentation, that is in the case
$M=1$ and $T=U$, we consider $q_t(x_t|\bx_{t-1}) = p_\theta(x_t|\bx_{t-1},\bY_U)$ and $w_t \equiv 1$.
Then by (\ref{f0mt}),
\begin{eqnarray*}
f_{2t}^c(\bx_t) [=f_{1,2t}^c(\bx_t)] & = & E_p(X_u|\bX_t=\bx_t,\bY_U)-E_p(X_u|\bY_U), \cr
& = & X_u - E_p(X_u|\bY_U) \mbox{ if } t \geq u. 
\end{eqnarray*}
Since ${\rm Var}_p(X_u|\bY_U) \geq {\rm Var}_p(X_u|X_{u-1}, X_{u+1}, Y_u) > 0$, 
therefore in this example, $\sigma^2_{c,2t}$ is bounded away from 0 for $t \geq u$,
and consequently $\sigma_c^2 \rightarrow \infty$ as $U \rightarrow \infty$.

Intuitively, in the case of the standard particle filter, 
the estimation of $E_p(X_u|\bY_U)$ is unstable as $U \rightarrow \infty$ simply due to degeneracy caused by repeated resamplings
at $t \geq u$. On the other hand, the repeated resamplings does not cause instability in the segmented method because resampling in one segment
does not result in sample depletion of another segment. 
There is a vast literature on smoothed latent-state estimators, see for example \cite{cappe}. 
We do not go into details here as our main motivation for looking at parallel particle filters is to achieve wall-clock computation time savings;  
the variance reductions in smoothed estimates can be viewed as an added benefit.

\subsubsection{Numerical Study} \label{sec:numerics}

As in Section 3.3.1, consider the linear time-series (\ref{lts}) and the estimation of $\psi(\bx_U) = x_u$ for $0 < u \leq U$,
conditioned on $\bY_U$. Kalman updating formulas are applied to
compute $E(X_u|\bY_{u-1})$ and $E(X_u|\bY_u)$ analytically, 
and the Rauch-Tung-Striebel smoother is applied to compute $E(X_u|\bY_t)$ for $t > u$.
The mean-squared errors (MSE) of the particle filter methods can then be computed using Monte Carlo. 

The first method we consider is the standard particle filter that performs
bootstrap resampling at every step. We select $q_1$ as  $N(0, \sigma_X^2)$ and $q_t(\cdot|\bx_{t-1})$ as
$N(ax_{t-1}, (1-a^2) \sigma_X^2)$ for $t>1$, hence
\begin{equation} \label{nwt}
w_t(\bx_t) = \exp \Big[-\frac{(Y_t-x_t)^2}{2 \sigma_Y^2} \Big] \quad t \geq 1.
\end{equation}

We next consider parallel particle filters with $\bX_U$ segmented into $M$ portions of equal length $T=U/M$.
The resampling weights are as in (\ref{nwt}), and like the standard particle filter,
$$q_{m,t}(\cdot|\bx_{t-1}) = N(ax_{t-1},(1-a^2) \sigma_X^2) \mbox{ for } (m-1)T+1 < t \leq mT. 
$$

\begin{table}[t] 
\begin{center}
\begin{tabular}{rr|lll}
& & \multicolumn{3}{|c}{MSE ($\times 10^{-2}$)} \cr
$u$ & $E_p(X_u|\bY_{50})$ & Standard & Seg.: $N(0,1)$ & Seg.: $N(\wht \mu_m,\wht \sigma_m^2)$ \cr \hline
5 & $-$0.46 & 5.8$\pm$1.0 & 0.67$\pm$0.09 & 0.6$\pm$0.1 \cr
10 & 0.45 & 6.0$\pm$0.8 & 0.35$\pm$0.04 & 0.34$\pm$0.05 \cr
15 & 0.80 & 5.1$\pm$0.6 & 0.64$\pm$0.08 & 0.58$\pm$0.07 \cr
20 & 0.95 & 4.9$\pm$0.7 & 1.1$\pm$0.2 & 0.77$\pm$0.09 \cr
25 & 1.16 & 6.1$\pm$1.0 & 2.8$\pm$0.4 & 3.0$\pm$0.4 \cr
30 & 2.81 & 2.7$\pm$0.4 & 2.2$\pm$0.3 & 1.0$\pm$0.1 \cr
35 & 0.81 & 1.8$\pm$0.2 & 1.9$\pm$0.2 & 0.8$\pm$0.1 \cr
40 & $-$1.10 & 1.4$\pm$0.2 & 1.5$\pm$0.2 & 0.9$\pm$0.2 \cr
45 & $-$0.50 & 0.7$\pm$0.1 & 0.69$\pm$0.09 & 0.55$\pm$0.07 \cr
50 & 0.38 & 0.19$\pm$0.02 & 0.18$\pm$0.02 & 0.20$\pm$0.03 \cr
\end{tabular}
\end{center}
\caption{ MSE $(\times 10^{-2})$ of $\wht E(X_u|\bY_{50})$ for: 1. the standard particle filter, 
2. the segmented particle filter initialized at $N(0,1)$, 
3. the segmented particle filter initialized at $N(\wht \mu_m,\wht \sigma_m^2)$,
with $(\wht \mu_m, \wht \sigma_m^2)$ estimated from past observations.}
\end{table}

We consider two versions of $q_{m,(m-1)T+1}$ ($=r_m$ here), the initial sampling distribution of 
$\wtd X_{m,(m-1)T+1}^k$. In the first version, we simply let $q_{m,(m-1)T+1} = N(0,\sigma_X^2)$.
In the second version, we let
$q_{m,(m-1)T+1} = N(\wht \mu_m,\wht \sigma_m^2)$, with $(\wht \mu_1,\wht \sigma_1^2)=(0,\sigma_X^2)$ and
for $m \geq 2$, $(\wht \mu_m,\wht \sigma_m^2)$ simulated using $(Y_{(m-1)T-r}, \ldots, Y_{(m-1)T})$ for 
some $r \geq 0$.
More specifically, we let 
\begin{eqnarray*}
\mu_m & = & E(X_{(m-1)T+1}|Y_{(m-1)T-r}, \ldots, Y_{(m-1)T}), \cr 
\sigma_m^2 & = & {\rm Var}(X_{(m-1)T+1}|Y_{(m-1)T-r}, \ldots, Y_{(m-1)T}), 
\end{eqnarray*}
and estimate them by sampling $(X_{(m-1)T-r}, \ldots, X_{(m-1)T+1})$ using particle filters. 

In our simulation study, we select $a=0.8$, $\sigma_X^2 = \sigma_Y^2=1$, $U=50$ and $M=5$.
We apply $K=500$ particles in each filter, and for the second version of the segmented method,
we consider $r=4$. Each method is repeated 100 times, and its MSE are reported in Table 1.
We see from Table 1 substantial variance reductions for parallel particle filters over standard particle filters,
especially when $U-u$ is big, agreeing with the discussions in Section 3.3.1. 
In addition, we see that applying estimation of $(\mu_m, \sigma_m^2)$ improves upon the 
performances of parallel particle filters. 

\subsubsection{In-sample variance estimation and particle size allocation}

Let $\sigma_{{\rm P}m}^2 = 
\sum_{u=2(m-1)T+1}^{2mT} \sigma_{c,u}^2$ be the variability attributed to the $m$th particle filter.
If $K_m$ particles are allocated to particle filter $m$ with $K_m$ large, 
then analogous to (\ref{CLT2}), 
\begin{equation} \label{Vtpsi}
{\rm Var}(\wtd \psi_U) \doteq \sum_{m=1}^M \frac{\sigma_{{\rm P} m}^2}{K_m}.
\end{equation}
Therefore being able to estimate $\sigma_{\rm{P}m}^2$ in-sample would allow us to optimally allocate the particle sizes 
in the $M$ particle filters so as to
minimize (\ref{Vtpsi}). The estimation can be done in the following manner.

Consider $1 \leq m \leq M$ and let $C_m^j = \{ k:A_{m,mT}^k=j \}$, noting that
$A_{m,mT}^k=j$ if and only if $\bX_{m,mT}^k$ is descended from $\wtd X_{(m-1)T+1}^j$. Let 
$$Q_j(\mu) = \sum_{\bk \in \mZ^M: k(m) \in C_m^j} L(\bX_U^\bk)[\psi(\bX_U^\bk)-\mu] H_U^\bk.
$$

\begin{thm} \label{thm3} Under the assumptions of Corollary \ref{cor1},
\begin{equation} \label{Qj}
\wht \sigma^2_{{\rm P}m}(\psi_U) := K^{-1} \sum_{j=1}^K Q_j^2(\psi_U) \srp \sigma_{{\rm P}m}^2,
\quad 1 \leq m \leq M.
\end{equation}
\end{thm}

Since $\wtd \psi_U \srp \psi_U$ by Corollary \ref{cor1}, therefore $\wht \sigma^2_{{\rm P}m}(\wtd \psi_U)$
is consistent for $\sigma^2_{{\rm P}m}$. Besides particle size allocation, being able to estimate $\sigma^2_{{\rm P}m}$, 
$1 \leq m \leq M$, and hence $\sigma_c^2$ allows
us to assess the level of accuracy of $\wtd \psi_U$ in estimating $\psi_U$.

\section{Discussion}\label{sec:summary}

We discuss here the subsampling approach, proposed in \cite{briers,fearnhead}, that can be used to reduce the
$\mathcal{O}(K^2)$ computational cost of our estimates.
We make the discussions more concrete here by considering $M=2$ and focusing on the estimation of the
likelihood $\lambda(\theta)$.

The actual computational cost of the double sum in (\ref{eq:lam_clar}) may be less expensive than it seems,
given that this operation is done only once, and that time-savings can be achieved if we bother to first group the segments
$\bX_{2,U}^\ell$ having a common first component. However asymptotically, we do have a
larger computational complexity due to the double sum. 

Let $\{ (k(v),\ell(v)): 1 \leq v \leq V \}$ be selected i.i.d. from $\beta$,  
a positive distribution on $\mZ_K^2$, and estimate $\lambda(\theta)$ by  
$$\wht \lambda^*(\theta) = \Big( \prod_{t=1}^U \bar w_t \Big) (K^2 V)^{-1}
\sum_{v=1}^V \frac{p_\theta(X_{T+1}^{\ell(v)}|X_T^{k(v)})}{r_2(X_{T+1}^{\ell(v)}) \beta(k(v), \ell(v))}.
$$
Since $\wht \lambda(\theta)$ is unbiased for $\lambda(\theta)$, therefore so is $\wht \lambda^*(\theta)$. 
For example, we can apply stratification sampling so that ``good" pairs are
chosen more frequently. The choice of $V \sim K^s$ for $s=1$ would give us a $\mathcal{O}(K)$ algorithm, 
though we may have to select $s>1$ in order to maintain the asymptotic variance of $\wht \lambda(\theta)$. 
As the computation of $\wht \lambda^*(\theta)$ is separate from the execution of the parallel particle filters
and can be done off-line, 
improving $\wht \lambda^*(\theta)$ with more sampling does not require re-running of the particle filters.

\subsubsection*{Acknowledgements}

The third author was supported by an MOE Singapore grant R-155-000-119-133 and is affiliated with the Risk Management Institute and Centre for Quantitative Finance at NUS.

\appendix

\section{Proofs}

We preface the proofs of the main results of Section 3 with 
two supporting lemmas in Appendix A.1 below. Lemma \ref{lem1} is a weak law of large number for sums of 
segmented sequences. Lemma \ref{lem2} provides finer approximations of such sums. 

 \subsection{Asymptotics and finer approximations for sums of segmented sequences}

\begin{lem} \label{lem1}
Let $G$ be a real-valued measurable function of $\bx_t$ for some $(m-1)T < t \leq mT$ with $1 \leq m \leq M$. 

\smallskip {\rm (a)} If $\wtd \mu_t := E_q[G(\bX_t)/h_{t-1}(\bX_{t-1})]$ exists and is finite, 
then as $K \rightarrow \infty$,
$$K^{-m} \sum_{\bk \in \mZ^m} G(\wbX_t^{\bk}) \srp \wtd \mu_t.
$$

{\rm (b)} If $\mu_t := E_q[G(\bX_t)/h_t(\bX_t)]$ exists and is finite, then as $K \rightarrow \infty$,
$$K^{-m} \sum_{\bk \in \mZ^m} G(\bX_t^{\bk}) \srp \mu_t. 
$$

{\rm (c)} For each $\bk \in \mZ^m$, 
\begin{equation} \label{HH}
\frac{\wtd H_t^{\bk}}{h_t(\wbX_t^{\bk})} = \frac{H_t^{\bk}}{h_t(\bX_t^{\bk})} =
\frac{\bar w_1 \cdots \bar w_t}{\eta_t} \srp 1.
\end{equation} 
\end{lem}

\begin{proof}
Since $G=G^+ - G^-$, we can assume without loss of generality that $G$ is nonnegative. The proofs of
(a) and (b) for $t \leq T$ follows from standard induction arguments, see \cite[Lemma 2]{CL13}. 
For $t>T$, induction arguments are again used, but the framework is now considerably
more complicated with summation on a multi-dimensional instead of a one-dimensional space.
Unlike in \cite{CL13}, characteristic functions are now needed in the induction proof. 

Let $T < u \leq 2T$ and assume that $\wtd \mu_u$ exists and is finite, and that
Lemma \ref{lem1}(b) holds for $t=u-1$. 
Let $V_u^\ell = K^{-1} \sum_{k=1}^K G(\wbX_u^{k \ell})$ and consider the decomposition 
$V_u^\ell = R_{u,c}^\ell + S_{u,c}^{\ell}$,
where
\begin{equation} \label{A2}
R_{u,c}^\ell = K^{-1} \sum_{k=1}^K G(\wbX_u^{k \ell}) {\bf 1}_{\{ G(\tilde \bX_u^{k \ell}) \leq c \}},
\ S_{u,c}^\ell = K^{-1} \sum_{k=1}^K G(\wbX_u^{k \ell}) {\bf 1}_{\{ G(\tilde \bX_u^{k \ell}) > c \}}.
\end{equation}
Let $\bar V_u = K^{-1} \sum_{\ell=1}^K V_u^\ell$, and define $\bar R_{u,c}$, $\bar S_{u,c}$ in a similar fashion.
Let $\bi = \sqrt{-1}$ and define
\begin{eqnarray} \label{tmu}
\wtd \mu_{u,c} & = & E_q[G(\bX_u) {\bf 1}_{\{ G(\bX_u) > c \}}/h_{u-1}(\bX_{u-1})], \\ \nonumber
\varphi^{\ell}_{2u-1,c}(\theta|\cF_{2u-2}) & = & E_K [ \exp ( \bi \theta K^{-1}  
R_{u,c}^\ell ) | \cF_{2u-2} ].
\end{eqnarray}
Let $\delta > 0$. 
Since $R_{u,\delta K}^1, \ldots, R_{u,\delta K}^K$ are independent conditioned on $\cF_{2u-2}$, 
\begin{eqnarray} \label{A3}
E_K [ \exp ( \bi \theta \bar V_u)| \cF_{2u-2} ] & = & 
\prod_{\ell=1}^K \varphi^{\ell}_{2u-1, \delta K}(\theta|\cF_{2u-2}) + r_K, \\ \nonumber
r_K & = & E_K \{ \exp ( \bi \theta \bar R_{u,\delta K} )  
[\exp ( \bi \theta \bar S_{u,\delta K} )-1 ] | \cF_{2u-2} \}.
\end{eqnarray}
Since $|e^{\bi z}-1| \leq |z|$,  for $K \geq c/\delta$,
\begin{equation} \label{rK1}
|r_K| \leq |\theta| K^{-1} \sum_{\ell=1}^K E_K(S_{u,c}^\ell | \cF_{2u-2}).
\end{equation}
By the induction hypothesis applied on
$G_c(\bx_{u-1}) := E_q[G(\bX_u) {\bf 1}_{\{ G(\bX_u) > c \}}|\bX_{u-1}=\bx_{u-1}]$,
\begin{equation} \label{rK2}
K^{-1} \sum_{\ell=1}^K E_K(S_{u,c}^\ell | \cF_{2u-2}) \stackrel{p}{\rightarrow} 
E_q[G_c(\bX_{u-1})/h_{u-1}(\bX_{u-1})]  (= \wtd \mu_{u,c}). 
\end{equation}
Since $\wtd \mu_{u,c} \rightarrow 0$ as $c \rightarrow \infty$, therefore by (\ref{rK1}) and (\ref{rK2}), 
\begin{equation} \label{A4}
r_K \srp 0.
\end{equation}

Let $\wtd R_{u,c} = R_{u,c}^\ell - E_K(R_{u,c}^\ell|\cF_{2u-2})$ and
\begin{equation} \label{tphi}
\wtd \varphi_{2u-1,c}^\ell(\theta|\cF_{2u-2}) = E_K[\exp(\bi \theta K^{-1} \wtd R_{u,c}^\ell)
|\cF_{2u-2}]. 
\end{equation}
Since $|E[e^{\bi \theta Z}-(1+\bi \theta Z-\frac{\theta^2 Z^2}{2})]| \leq \theta^2 E Z^2$ (see
(26.5) of \cite{Bil91}) and $(R_{u,\delta K}^\ell)^2 \leq \delta K R_{u,\delta K}^\ell$,
\begin{eqnarray} \label{Bill}
& & |\wtd \varphi_{2u-1,\delta K}^\ell(\theta|\cF_{2u-2})- \{ 1-[\theta^2/(2K^2)] {\rm Var}_K(R_{u,\delta K}^\ell|\cF_{2u-2}) \} | \\
\nonumber & \leq & (\theta/K)^2 {\rm Var}_K(R_{u,\delta K}^\ell|\cF_{2u-2}) \leq (\theta/K)^2 E_K[(R_{u,\delta K}^\ell)^2|\cF_{2u-2}] 
\\ \nonumber
& \leq & (\delta
\theta^2/K) E_K(R_{u,\delta K}^\ell|\cF_{2u-2}).
\end{eqnarray}
Since $|\prod_{\ell=1}^K z_\ell - \prod_{\ell=1}^K y_\ell| \leq \sum_{\ell=1}^K |z_\ell-y_\ell|$ whenever
$|z_\ell| \leq 1$ and $|y_\ell| \leq 1$ for all $\ell$ (see Lemma 1 in Section 2.7 of \cite{Bil91}),
by the induction hypothesis applied on $G_c'(\bx_u) = E_q[G(\bX_u) {\bf 1}_{\{ G(\bX_u) \leq c \}}|
\bX_{u-1}=\bx_{u-1}]$ and then letting $c \rightarrow \infty$, 
\begin{eqnarray} \label{4.11} 
& & \Big| \prod_{\ell=1}^K \wtd \varphi_{2u-1,\delta K}^\ell(\theta|\cF_{2u-2})-
\prod_{\ell=1}^K \{ 1-[\theta^2/(2K^2)] {\rm Var}_K(R_{u,\delta K}^\ell|\cF_{2u-2}) \} \Big| \\ \nonumber
& \leq & \delta \theta^2 K^{-1} \sum_{\ell=1}^K 
E_K(R_{u,\delta K}^\ell|\cF_{2u-2}) \srp \delta \theta^2 \wtd \mu_u.
\end{eqnarray}

Let $\delta_0 > 0$ be such that $\log(1-y) \geq -2y$ for $0 < y < (\theta \delta_0)^2$.
Therefore by the inequalities in (\ref{Bill}), for $0 < \delta \leq \delta_0$, 
\begin{eqnarray} \label{4.12}
& & \prod_{\ell=1}^K \{ 1 - [\theta^2/(2K^2)] {\rm Var}_K(R_{u,\delta K}^\ell|\cF_{2u-2}) \} \\ \nonumber 
& \geq & \prod_{\ell=1}^K \{ 1-[\delta \theta^2/(2K)] E_K[(R_{u,\delta K}^\ell)^2|\cF_{2u-2}] \} \\ \nonumber
& \geq & \exp \Big[ -\delta \theta^2 K^{-1} \sum_{\ell=1}^K E_K(R_{u,\delta K}^\ell|\cF_{2u-2}) \Big] 
\srp \exp(-\delta \theta^2 \wtd \mu_u).
\end{eqnarray}
By the definitions of $\varphi_{2u-1,c}^\ell$ and $\wtd \varphi_{2u-1,c}^\ell$ in (\ref{tmu}) and (\ref{tphi}),
\begin{eqnarray} \label{phirat}
& & \prod_{\ell=1}^K [\varphi_{2u-1,\delta K}^\ell(\theta|\cF_{2u-2})/\wtd \varphi_{2u-1,\delta K}^\ell(\theta|\cF_{2u-2})] \\
\nonumber & = & \exp \Big[ \bi \theta K^{-1} \sum_{\ell=1}^K E_K(R_{u,\delta K}^\ell|\cF_{2u-1}) \Big]
\srp \exp(\bi \theta \wtd \mu_u).
\end{eqnarray}
It follows from (\ref{A3}), (\ref{A4}) and (\ref{4.11})--(\ref{phirat}), with $\delta \rightarrow 0$, that
$$E_K[\exp(\bi \theta \bar V_u)|\cF_{2u-2}] \stackrel{p}{\rightarrow} \exp(\bi \theta \wtd \mu_u).
$$
Therefore $E_K \exp(\bi \theta \bar V_u) \stackrel{p}{\rightarrow} \exp(\bi \theta \wtd \mu_u)$, equivalently
$\bar V_u \stackrel{p}{\rightarrow} \wtd \mu_u$. Hence
(a) holds for $t=u$ whenever (b) holds for $t=u-1$. By similar arguments, (b) holds for $t=u$ whenever
(a) holds for $t=u$. The induction arguments to show (a) and (b) for $T < t \leq 2T$ are now complete.
Similar induction arguments can be used to show (a) and (b) for $(m-1)T < t \leq mT$ for $m=3, \ldots, M$.

The identities in (\ref{HH}) follows from multiplying (\ref{Ht}) over ``$(m,t)$" $=(1,T), \ldots, (m-1,(m-1)T), (m,t)$. 
By (\ref{eta}) and (\ref{h1}), applying (a) on $G=w_t$ gives us $\bar w_t \srp
\eta_t/\eta_{t-1}$. Therefore (c) holds. 
\end{proof}

\begin{lem} \label{lem2}
Let $G_u$ be a measurable function of $\bx_u$ with $(m-1)T < u \leq mT$ and define 
$G_{m,u}(\bx_{m,u}) = E_q[G_u(\bX_u)|\bX_{m,u}=\bx_{m,u}]$. 

\smallskip {\rm (a)} If $E_q[G_u^2(\bX_u) h_{u-1}(\bX_{u-1})] < \infty$, then
\begin{eqnarray} \label{2a1}
& & K^{-1} \sum_{\ell=1}^K \Big[ K^{-M+1} \sum_{\bk \in \mZ^{m-1}} G_u(\wbX_u^{\bk \ell}) H_{u-1}^{\bk \ell}
-G_{m,u}(\wbX_{m,u}^\ell) H_{m,u-1}^{\ell} \Big]^2 \srp 0, \\
\label{2a2} 
& & \sum_{\ell=1}^K \Big[ K^{-M+1} \sum_{\bk \in \mZ^{m-1}} G_u(\wbX_u^{\bk \ell}) H_{u-1}^{\bk \ell}
- G_{m,u}(\wbX_{m,u}^\ell) H_{m,u-1}^{\ell} \Big] = o_p(K^{\frac{1}{2}}), 
\end{eqnarray}
where $\bk \ell=  (k(1), \ldots, k(m-1), \ell)$.

{\rm (b)} If $E_q[G_u^2(\bX_u) h_u(\bX_u)] < \infty$, then
\begin{eqnarray} \label{2b1}
& & K^{-1} \sum_{\ell=1}^K \Big[ K^{-M+1} \sum_{\bk \in \mZ^{m-1}} G_u(\bX_u^{\bk \ell}) H_u^{\bk \ell}
-G_{m,u}(\bX_{m,u}^\ell) H_{m,u}^{\ell} \Big]^2 \srp 0, \\
\label{2b2} 
& & \sum_{\ell=1}^K \Big[ K^{-M+1} \sum_{\bk \in \mZ^{m-1}} G_u(\bX_u^{\bk \ell}) H_u^{\bk \ell}
-G_{m,u}(\bX_{m,u}^\ell) H_{m,u}^{\ell} \Big] = o_p(K^{\frac{1}{2}}). 
\end{eqnarray}
\end{lem}

\begin{proof}
The case $m=1$ is trivial, so let us first consider $m=2$. For $1 \leq t \leq T$, let
$$G_{t,u}(\bx_t,\bx_{2,u}) = E_q[G_u(\bX_u)|\bX_t=\bx_t, \bX_{2,u}=\bx_{2,u}].
$$
By Lemma \ref{lem1}(c), $K^{-1} \sum_{k=1}^K G_u(\wbX_u^{k \ell}) H_{u-1}^{k \ell}
-G_{2,u}(\wbX_{2,u}^\ell) H_{2,u-1}^\ell =[1+o_p(1)] D^\ell$ uniformly over $\ell$, where 
\begin{equation} \label{l23}
D^\ell = \Big[ K^{-1} \sum_{k=1}^K G_u(\wbX_u^{k \ell}) H_T^k 
-G_{2,u}(\wbX_{2,u}^\ell) \Big] h_{2,u-1}(\bX_{2,u-1}^\ell).
\end{equation}
By (\ref{l23}), we have the expansion
\begin{eqnarray} \label{Dell}
D^\ell & = & K^{-1} \sum_{s=1}^{2T} (d_s^{1 \ell} + \cdots + d_s^{K \ell}) h_{2,u-1}(\bX_{2,u-1}^\ell), \\ \nonumber
\mbox{where } d_{2t-1}^{k \ell} & = & [G_{t,u}(\wbX_t^k,\wbX_{2,u}^\ell)-G_{t-1,u}(\bX_{t-1}^k,\wbX_{2,u}^\ell)]
H_{t-1}^k, \\ \nonumber
d_{2t}^{k \ell} & = & G_{t,u}(\bX_t^k,\wbX_{2,u}^\ell) H_t^k - \sum_{j=1}^K W_t^j
G_{t,u}(\wbX_t^k,\wbX_{2,u}^\ell) H_{t-1}^k, 
\end{eqnarray}
with the convention that for $t=1$, $H_{t-1}^k=1$ and $G_{t-1,u}(\bX_{t-1}^k, \wbX_{2,u}^\ell) = 
G_{2,u}(\wbX_{2,u}^\ell)$. 

Let $D_v^\ell = K^{-1} \sum_{s=1}^v (d_s^{1 \ell} + \cdots + d_s^{K \ell}) h_{2,u-1}(\bX_{2,u-1}^\ell)$.
We shall show inductively that uniformly over $\ell$, 
\begin{equation} \label{indD}
\sum_{\ell=1}^K E_K[(D_v^\ell)^2|\cG_{2u-2}] = O_p(1), \quad v=1, \ldots, 2T,
\end{equation}
where $\cG_{2u-2}$ denotes the $\sigma$-algebra for all random variables generated in the second particle filter
up to and including resampling at the $(u-1)$th stage. 
Since $d_1^{1 \ell}, \ldots, d_1^{K \ell}$ are uncorrelated with mean 0 conditioned on
$\cG_{2u-2}$, by Lemma \ref{lem1}(a) and (c), 
$$\sum_{\ell=1}^K E_K[(D_1^\ell)^2|\cG_{2u-2}] = K^{-2} \sum_{k=1}^K \sum_{\ell=1}^K
E_K[(d_1^{k \ell})^2|\cG_{2u-2}] h_{2,u-1}(\bX_{2,u-1}^\ell)^2 = O_p(1). 
$$
Therefore (\ref{indD}) holds for $v=1$. Consider next $v>1$. Let $\cH_v = \cF_v \cup \cG_{2u-2}$. Since
$$D_v^\ell = D_{v-1}^\ell + K^{-1} (d_v^{1 \ell} + \cdots d_v^{K \ell}) h_{2,u-1}(\bX_{2,u-1}^\ell),
$$
and $d_v^{1 \ell}, \ldots, d_v^{K \ell}$
are conditionally independent with mean 0 given $\cH_{v-1}$, by Lemma \ref{lem1}(a) and (c),
\begin{eqnarray*}
\sum_{\ell=1}^K E_K[(D_v^\ell)^2|\cH_{v-1}] & = & \sum_{\ell=1}^K
(D_{v-1}^\ell)^2 + K^{-2} \sum_{k=1}^K \sum_{\ell=1}^K E_K[(d_v^{k \ell})^2|\cH_{v-1}] h_{2,u-1}^2(\bX_{2,u-1}^\ell) \cr
& = & \sum_{\ell=1}^K (D_{v-1}^\ell)^2 + O_p(1).
\end{eqnarray*}
Therefore (\ref{indD}) for $v$ follows from (\ref{indD}) for $v-1$. 
By induction, (\ref{indD}) holds for $1 \leq v \leq 2T$.
In particular, since $D_{2T}^\ell = D^\ell$, (\ref{2a1}) holds for $m=2$.

By (\ref{indD}) for $v=2T$, and noting that 
$D^1, \ldots, D^K$ are conditionally independent with mean 0 given $\cF_{2u-2}$,
and that $\cG_{2u-2} \subset \cF_{2u-2}$,
\begin{equation} \label{EKD}
E_K \Big[ \Big( \sum_{\ell=1}^K D^\ell \Big)^2 \Big| \cF_{2u-2} \Big] 
= \sum_{\ell=1}^K E_K[(D^\ell)^2|\cF_{2u-2}] = O_p(K).
\end{equation}
Therefore $K^{-1} (\sum_{\ell=1}^K D^\ell)^2 \srp 0$, and (\ref{2a2}) holds for $m=2$. 
The extension of the proof to $m>2$ and 
the proofs of (\ref{2b1}) and (\ref{2b2}) apply similar arguments to those of (\ref{2a1})
and (\ref{2a2}).
\end{proof}

\subsection{Proofs of Theorem \ref{thm2}, Corollary \ref{cor1} and Theorem \ref{thm3}} 

\begin{proof}[Proof of Theorem \ref{thm2}]
 Let $S=\sum_{m=1}^M S_m$, where
\begin{eqnarray*}
S_m & = & \sum_{u=2(m-1)T+1}^{2mT} (Z_{m,u}^1 + \cdots +Z_{m,u}^K), \cr
Z_{m,2t-1}^k & = & [f_{m,t}(\wbX_{m,t}^k)-f_{m,t-1}(\bX_{m,t-1}^k)] H_{m,t}^k, \cr
Z_{m,2t}^k & = & f_{m,t}(\bX_{m,t}^k) H_{m,t}^k-\sum_{j=1}^K W_t^j f_{m,t}(\wbX_{m,t}^j) \wtd H_{m,t}^j.
\end{eqnarray*}
By the CLT for particle filters on unsegmented HMM sequences, 
see for example \cite{CL13}, 
$$K^{-\frac{1}{2}} S_m \Rightarrow N(0,\sigma^2_{\rm Pm}), \quad 1 \leq m \leq M. 
$$
Since the particle filters operate independently and $S_m$ depends only on
the $m$th particle filter, 
\begin{equation} \label{KS1}
K^{-\frac{1}{2}} S \Rightarrow N(0,\sigma^2) \mbox{ where } \sigma^2 = \sum_{m=1}^M \sigma^2_{{\rm P}m}.
\end{equation}
By (\ref{MD2}),
\begin{equation} \label{KS2}
\sqrt{K}(\wht \psi_U - \psi_U) = K^{-\frac{1}{2}} S + K^{-\frac{1}{2}} \sum_{u=1}^{2mU} \sum_{\ell=1}^K
(Z_u^\ell-Z_{m,u}^\ell).
\end{equation}
Therefore by (\ref{Z2}), (\ref{Z22}) and (\ref{2a2}) applied on $G_t(\bx_t)=f_t(\bx_t)-f_{t-1}(\bx_{t-1})$, 
\begin{equation} \label{Zd}
K^{-\frac{1}{2}} \sum_{k=1}^K
(Z_u^k - Z_{m,u}^k) \srp 0
\end{equation}
for $u=2t-1$. And by (\ref{Z2}), (\ref{Z22}) and (\ref{2b2}) applied on $G_t(\bx_t)=f_t(\bx_t)$ and
(\ref{2b2}) applied on $G_t(\bx_t) = w_t(\bx_t) f_t(\bx_t)$, (\ref{Zd}) holds for $u=2t$.
We conclude Theorem \ref{thm2} from (\ref{KS1})--(\ref{Zd}). 
\end{proof}

\begin{proof}[Proof of Corollary \ref{cor1}]
 By (\ref{tpsi}), 
\begin{equation} \label{PC1}
\wtd \psi_U - \psi_U = \Big[ K^{-M} \sum_{\bk \in \mZ^M} L(\bX_U^{\bk}) H(\bX_U^{\bk}) \Big]^{-1}
\wht \psi^c_U,
\end{equation}
where $\wht \psi^c_U = K^{-M} \sum_{\bk \in \mZ^M} L(\bX_U^{\bk}) \psi^c(\bX_U^{\bk}) H_U^{\bk}$, 
is (\ref{hatpsi}) with $\psi$ replaced by $\psi^c(\bx_U) = \psi(\bx_U) - \psi_U$. 
By Theorem \ref{thm2},
\begin{equation} \label{PC2}
\sqrt{K} \wht \psi^c_U \Rightarrow N(0,\sigma_c^2).
\end{equation}
By Lemma \ref{lem1}(b) and (c), 
$$K^{-M} \sum_{\bk \in \mZ^M} L(\bX_U^{\bk}) H(\bX_U^{\bk}) \srp 1, 
$$
and Corollary \ref{cor1} therefore follows from (\ref{PC1}) and (\ref{PC2}). 
\end{proof}

\begin{proof}[Proof of Theorem \ref{thm3}]. 
We shall show Theorem \ref{thm3} in detail for the case $M=m=2$, assuming without loss of generality that $\psi_U=0$.
It follows from \cite[Corollary 2]{CL13} that  
\begin{equation} \label{ej}
K^{-1} \sum_{j=1}^K (e^j)^2 \srp \sigma^2_{\rm P2}, \mbox{ where } 
e^j = \sum_{\ell: A_{2,U}^\ell=j} f_{2,U}(\bX_{2,U}^\ell) H_{2,U}^\ell. 
\end{equation}
By (\ref{Z2}), (\ref{Z22}) and (\ref{Qj}),
$$\wht \sigma^2_{{\rm P}2}(0) = K^{-1} \sum_{j=1}^K \Big( e^j + \sum_{v=2T+1}^{2U}
\sum_{\ell:A_{2,t}^\ell=j} \zeta_v^\ell \Big)^2, 
$$
where $t=\lfloor v/2 \rfloor$, with $\lfloor \cdot \rfloor$ denoting the greatest integer function, 
\begin{eqnarray*}
\zeta_{2u-1}^\ell & = & Z_{2u-1}^\ell - Z_{2,2u-1}^\ell, \quad \zeta_{2u}^\ell = (\#_u^\ell - KW_u^\ell) \chi^\ell, \cr
\chi^\ell & = & K^{-1} \sum_{k=1}^K f_u(\wbX^{k \ell}) \wtd H_u^{k \ell} - f_{2,u}(\bX_{2,u}^\ell)
\wtd H_{2,u}^\ell.
\end{eqnarray*}
Therefore by (\ref{ej}), to show (\ref{Qj}), it suffices to show that
\begin{equation} \label{Kzeta}
K^{-1} \sum_{j=1}^K \Big( \sum_{v=2T+1}^{2U} \sum_{\ell:A_{2,t}^\ell=j} \zeta_v^\ell \Big)^2
\srp 0.
\end{equation}
We shall apply induction to show that 
\begin{equation} \label{zeta}
K^{-1} \sum_{j=1}^K \Big( \sum_{v=2T+1}^{s} \sum_{\ell:A_{2,t}^\ell=j} \zeta_v^\ell \Big)^2
\srp 0, \quad s=2T+1, \ldots, 2U.
\end{equation}

By (\ref{Z2}), (\ref{Z22}) and (\ref{l23}), 
$\zeta_{2u-1}^\ell = [1+o_p(1)] D^\ell$ uniformly in $\ell$, for $G_u(\bx_u)=f_u(\bx_u)-f_{u-1}(\bx_{u-1})$.
Therefore by (\ref{EKD}),  
\begin{equation} \label{eps0}
\sum_{\ell=1}^K E_K[(\zeta_{2u-1}^\ell)^2|\cF_{2u-2}] \srp 0. 
\end{equation}
Since $A_{2,T}^\ell=\ell$, $1 \leq \ell \leq K$, (\ref{eps0}) for $u=T+1$ implies (\ref{zeta}) for $s=2T+1$. 
Since $\zeta_{2u-1}^1, \ldots, \zeta_{2u-1}^K$ are independent with mean 0 conditioned on $\cF_{2u-2}$,
and $\zeta_v^\ell$ are measurable with respect to $\cF_{2u-2}$ for $v \leq 2u-2$,
\begin{eqnarray*}
& & K^{-1} \sum_{j=1}^K E_K \Big[ \Big( \sum_{v=2T+1}^{2u-1} \sum_{\ell:A^\ell_{2,t}=j} \zeta_v^\ell \Big)^2 \Big|
\cF_{2u-2} \Big] \cr
& = & K^{-1} \sum_{j=1}^K \Big( \sum_{v=2T+1}^{2u-2} \sum_{\ell:A^\ell_{2,t}=j} \zeta_v^\ell \Big)^2 + K^{-1}
\sum_{\ell=1}^K E_K[(\zeta_{2u-1}^\ell)^2|\cF_{2u-2}].
\end{eqnarray*}
Therefore by (\ref{eps0}), (\ref{zeta}) for $s=2u-1$ follows from (\ref{zeta}) for $s=2u-2$. 

Since
${\rm Var}_K(\#_u^\ell|\cF_{2u-1}) = KW_u^\ell(1-W_u^\ell)$, ${\rm Cov}_K(\#_u^i,\#_u^\ell|\cF_{2u-1})=
-KW_u^i W_u^\ell$, and $\chi^\ell$ are measurable with respect to $\cF_{2u-1}$, 
\begin{eqnarray} \label{evenv}
& & K^{-1} \sum_{j=1}^K \Big[ \Big( \sum_{v=2T+1}^{2u} \sum_{\ell:A_{2,T}^\ell=j} \zeta_v^\ell \Big)^2 \Big| \cF_{2u-1} \Big] \\
\nonumber & = & K^{-1} \sum_{j=1}^K \Big( \sum_{v=2T+1}^{2u-1} \sum_{\ell:A_{2,T}^\ell=j} \zeta_v^\ell \Big)^2
+ \sum_{\ell=1}^K W_u^\ell (\chi^\ell)^2 - K \Big( \sum_{\ell=1}^K W_u^\ell \chi^\ell \Big)^2.
\end{eqnarray}
It follows from an induction argument similar to that used to show (\ref{eps0}) that 
$$\sum_{\ell=1}^k W_u^\ell (\chi^\ell)^2 \srp 0.
$$
Therefore by (\ref{evenv}), (\ref{zeta}) for $s=2u$ follows from (\ref{zeta}) for $s=2u-1$. 
The induction arguments are complete and we have shown that (\ref{zeta}) holds for $2T+1 \leq s \leq 2U$,
and Theorem \ref{thm3} holds.
\end{proof}

\end{document}